\documentclass[leqno,12pt]{article}
\usepackage{amsfonts}
\usepackage{amsmath, amssymb}
\usepackage{amsthm}

\theoremstyle{plain}
\newtheorem{theorem}{\indent\sc Theorem}[section]

\newtheorem{corollary}[theorem]{\indent\sc Corollary}
\newtheorem{proposition}[theorem]{\indent\sc Proposition}

\theoremstyle{definition}

\newtheorem{remark}[theorem]{\indent\sc Remark}
\newtheorem{example}[theorem]{\indent\sc Example}

\newcommand\on{\operatorname}
\renewcommand\div{\on{div}}
\newcommand\Hess{\on{Hess}}
\newcommand\Ric{\on{Ric}}
\newcommand\scal{\on{scal}}
\newcommand\vol{\on{vol(M)}}

\newcommand\func{\operatorname}
\newcommand\grad{\func{grad}}

\begin{document}

\title{Remarks on almost Riemann solitons \\with gradient or torse-forming vector field}
\author{Adara M. Blaga}
\date{}
\maketitle

\begin{abstract}
We consider almost Riemann solitons $(V,\lambda)$ in a Riemannian manifold and underline their relation to almost Ricci solitons. When $V$ is of gradient type, using Bochner formula, we explicitly express the function $\lambda$ by means of the gradient vector field $V$ and illustrate the result with suitable examples. Moreover, we deduce some properties for the particular cases when the potential vector field of the soliton is solenoidal or torse-forming, with a special view towards curvature.
\end{abstract}

\markboth{{\small\it {\hspace{2cm} Remarks on almost Riemann solitons}}}{\small\it{Remarks on almost Riemann solitons
\hspace{2cm}}}


\bigskip

\section{Introduction}

Ricci solitons and Riemann solitons correspond to self-similar solutions of Ricci flow and Riemann flow, respectively,
introduced in 1982 \cite{ham-82} by R. S. Hamilton. The intrinsic geometric flows, \textit{Ricci flow}
\begin{equation}
\frac{\partial }{\partial t}g(t)=-2\Ric(g(t))
\end{equation}
and \textit{Riemann flow}
\begin{equation}
\frac{\partial }{\partial t}G(t)=-2R(g(t)),
\end{equation}
where $G:=\frac{1}{2}g\odot g$, for $\odot$ the Kulkarni-Nomizu product, $\Ric$ the Ricci curvature tensor and $R$ the Riemann curvature tensor of $g$, are two evolution equations formally similar to the heat equation, with applications in various fields. Extending the notion of Ricci flow to a nonlinear PDE which involves the Riemann curvature tensor, the Riemann flow has very similar properties to that of the Ricci flow \cite{ud}.

On a given $n$-dimensional smooth manifold $M$, a Riemannian metric $g$ and a non-vanishing vector field $V$ is said to define \textit{a Ricci soliton} \cite{ham} if there exists a real constant $\lambda$ such that
\begin{equation}\label{1}
\frac{1}{2}\pounds _{V}g+\func{Ric}=\lambda g,
\end{equation}
respectively, \textit{a Riemann soliton} \cite{hi} if there exists a real constant $\lambda$ such that
\begin{equation}\label{2}
\frac{1}{2}\pounds _{V}g\odot g+R=\lambda G,
\end{equation}
where $G:=\frac{1}{2}g\odot g$, $\pounds _{V}$ denotes the Lie derivative operator in the direction of the vector field $V$, $R$ and $\Ric$ are the Riemann and the Ricci curvature of $g$, respectively.
If $\lambda$ is a smooth function on $M$, we call $(V,\lambda)$ an \textit{almost Ricci} or an \textit{almost Riemann soliton}. Moreover, if $V$ is a gradient vector field, we call $(V,\lambda)$ a \textit{gradient almost Ricci} or a \textit{gradient almost Riemann soliton}. A soliton defined by $(V,\lambda)$ is said to be \textit{shrinking}, \textit{steady} or \textit{expanding} according as $\lambda$ is positive, zero or negative, respectively.

Ricci solitons are natural generalizations of Einstein manifolds and Riemann solitons are natural generalizations of spaces of constant sectional curvature. It was proved that on a compact Riemannian manifold, Ricci and Riemann solitons are gradient solitons (see \cite{pe}, \cite{hi}).

Note that the relations between Riemann flow and Ricci flow have been studied in \cite{ud} and some geometric properties of Riemann solitons in almost contact geometry, precisely, in Kenmotsu, Sasakian and $K$-contact manifolds, have been given in \cite{uday}, \cite{de}, \cite{v}.

In the present paper, we study some properties of almost Riemann and almost Ricci solitons, providing some relations between them. A central result is Theorem \ref{p1}, where for a gradient almost Riemann soliton, we explicitly compute the function $\lambda$ by means of $V$, without making all the laborious curvature computations. When $V$ is a solenoidal vector field of gradient type, we express the volume of a compact Riemannian manifold in terms of the Ricci tensor's norm. Moreover, we prove that any gradient Riemann soliton with potential vector field $V$ of constant length is steady and $V$ is a solenoidal vector field. Remark that the solenoidal vector fields of gradient type are the gradients of harmonic functions, and it is known that they are relevant in electrostatic.
We also consider the case when the potential vector field of the soliton is torse-forming and, in particular, we characterize the Ricci symmetric, Ricci semisymmetric manifolds admitting almost Riemann solitons with concircular vector field.

\section{Almost Riemann and almost Ricci solitons}

The Kulkarni-Nomizu product for two $(0,2)$-tensor fields $T_1$ and $T_2$ defined by:
$$(T_1 \odot T_2)(X,Y,Z,W):=T_1(X,W)T_2(Y,Z)+T_1(Y,Z)T_2(X,W)$$$$-T_1(X,Z)T_2(Y,W)-T_1(Y,W)T_2(X,Z),$$
for any $X,Y,Z,W\in \mathfrak{X} (M)$.
Then the Riemann soliton equation (\ref{2}) is explicitly expressed as
$$
2R(X,Y,Z,W)+[g(X,W)(\pounds _{V}g)(Y,Z)+g(Y,Z)(\pounds _{V}g)(X,W)
$$
$$-g(X,Z)(\pounds _{V}g)(Y,W)-g(Y,W)(\pounds _{V}g)(X,Z)]$$
\begin{equation}\label{3}=2\lambda [g(X,W)g(Y,Z)-g(X,Z)g(Y,W)],
\end{equation}
which by contraction over $X$ and $W$, gives
\begin{equation}\label{4}
\frac{1}{2}\pounds _{V}g+\frac{1}{n-2}\Ric=\frac{(n-1)\lambda-\div(V)}{n-2}g
\end{equation}
and further
\begin{equation}\label{9}
\scal=(n-1)[n\lambda-2\div(V)],
\end{equation}
provided $n\geq 3$.

If $n=2$, then (\ref{3}) implies that $M$ is an Einstein manifold with the Ricci curvature $\Ric=[\lambda-\div(V)]g$ and of scalar curvature $\scal=2[\lambda-\div(V)]$.

\begin{remark}
From the decomposition of the Riemann curvature tensor with respect to the Weil curvature tensor $\mathcal{W}$,
$$R=\mathcal{W}-\frac{\scal}{2(n-1)(n-2)}g\odot g+\frac{1}{n-2}\Ric\odot g,$$
if $(V,\lambda)$ defines an almost Riemann soliton on the $n$-dimensional Riemannian manifold $(M,g)$, $n\geq 3$, then the Weil curvature tensor vanishes. Note that properties of the Ricci solitons with vanishing Weil curvature have been given in \cite{ca}.
\end{remark}

Now, if we assume that $V$ is a solenoidal vector field, i.e. $\div(V)=0$, the equation (\ref{4}) becomes
$$\pounds _{V}g+\frac{2}{n-2}\Ric=2\frac{n-1}{n-2}\lambda\cdot g$$
and the scalar curvature equals to
$$\scal=n(n-1)\lambda,$$
so we can state:

\begin{proposition}
Let $(V,\lambda)$ define an almost Riemann soliton on an $n$\text{-}di\-men\-sio\-nal Riemannian manifold $(M,g)$, $n\geq 3$, with $V$ a solenoidal vector field. Then

i) $(V,\bar{\lambda})$, where $\bar{\lambda}=\frac{n-1}{n-2}\lambda$, defines an almost $\alpha$-Ricci soliton \cite{kk}, for $\alpha=\nolinebreak\frac{1}{n-2}$;

ii) $(\bar{V},\bar{\lambda})$, where $\bar{V}=(n-2)V$ and $\bar{\lambda}=(n-1)\lambda$, defines an almost Ricci soliton.
\end{proposition}

\subsection{Gradient solitons}

In \cite{blag}, using the Bochner formula, we have proved that for any gradient vector field $V\in \mathfrak{X}(M)$, we have:
\begin{equation}\label{15}
\Delta(|V|^2)-2|\nabla V|^2=2\Ric(V,V)+2V(\div(V))
\end{equation}
and
\begin{equation}\label{12}
(\div(\pounds _{V}g))(V)=2V(\div(V))+2\Ric(V,V).
\end{equation}

Taking the divergence of (\ref{4}), we get:
\begin{equation}\label{10}
\frac{1}{2}\div(\pounds _{V}g)+\frac{1}{n-2}\div(\Ric)=\frac{(n-1)d\lambda-d(\div(V))}{n-2}.
\end{equation}

Differentiating (\ref{9}), we have:
\begin{equation}\label{11}
d(\scal)=(n-1)[nd\lambda-2d(\div(V))].
\end{equation}

Taking into account that $\div(\Ric)=\frac{1}{2}d(\scal)$, from (\ref{10}) and (\ref{11}), we obtain:
$$(\div(\pounds _{V}g))(V)=2V(\div(V))-(n-1)V(\lambda)$$
and using (\ref{12}), we get:
\begin{equation}\label{13}
\Ric(V,V)=-\frac{n-1}{2}V(\lambda).
\end{equation}

Computing now (\ref{4}) in $(V,V)$ and using (\ref{13}), we obtain:
\begin{equation}\label{14}
\div(V)=(n-1)\lambda+\frac{n-1}{2}\cdot \frac{V(\lambda)}{|V|^2}-\frac{n-2}{2}\cdot\frac{V(|V|^2)}{|V|^2}.
\end{equation}

From (\ref{14}), (\ref{13}) and (\ref{15}) we explicitly determine $\lambda$ by means of $V$:

\begin{theorem}\label{p1}
Let $(V,\lambda)$ define a gradient almost Riemann soliton on the $n$\text{-}di\-men\-sio\-nal Riemannian manifold $(M,g)$, $n\geq 3$. Then
$$\lambda=\frac{1}{2(n-1)|V|^2}[\Delta(|V|^2)-2|\nabla V|^2+(n-2)V(|V|^2)-2V(\div(V))]\linebreak[1]+\frac{1}{n-1}\div(V).$$
\end{theorem}

\begin{example}
On the $3$-dimensional manifold $M=\{(x,y,z)\in\mathbb{R}^3, z>0\}$, where $(x,y,z)$ are the standard coordinates in $\mathbb{R}^3$, with the Riemannian metric
$$g:=\frac{1}{z^2}(dx\otimes dx+dy\otimes dy+dz\otimes dz),$$ the pair $(V=\frac{\partial}{\partial z}, \lambda=-\frac{2}{z}-1)$ defines an expanding gradient almost Riemann soliton. Precisely, $V=\grad(f)$, for $f(x,y,z)=-\frac{1}{z}$ and one can check that $|V|^2=\frac{1}{z^2}$, $V(|V|^2)=-\frac{2}{z^3}$, $\Delta(|V|^2)=\frac{8}{z^2}$, $|\nabla V|^2=\frac{3}{z^2}$, $\div(V)=-\frac{3}{z}$, $V(\div(V))=\frac{3}{z^2}$ therefore, $\lambda=-\frac{2}{z}-1$ is obtained from Theorem \ref{p1}, without making all the laborious curvature computations.
\end{example}

\begin{example}
On the $3$-dimensional manifold $M=\{(x,y,z)\in\mathbb{R}^3, z>1\}$, where $(x,y,z)$ are the standard coordinates in $\mathbb{R}^3$, with the Riemannian metric
$$g:=e^{2z}(dx\otimes dx+dy\otimes dy)+dz\otimes dz,$$
the pair $(V=e^z\frac{\partial}{\partial z}, \lambda=2e^z-1)$ defines a shrinking gradient almost Riemann soliton. Precisely, $V=\grad(f)$, for $f(x,y,z)=e^z$ and one can check that $|V|^2=e^{2z}$, $V(|V|^2)=2e^{3z}$, $\Delta(|V|^2)=8e^{2z}$, $|\nabla V|^2=3e^{2z}$, $\div(V)=3e^z$, $V(\div(V))=3e^{2z}$ therefore, $\lambda=2e^z-1$ is immediately obtained from Theorem \ref{p1}.
\end{example}

\begin{remark}
Similarly, if $(V,\lambda)$ defines a gradient almost Ricci soliton, then, in the same way, we can explicitly express the function $\lambda$ in terms of $V$, namely
\begin{equation}\label{e5}
\lambda=\frac{1}{2|V|^2}[\Delta(|V|^2)-2|\nabla V|^2+V(|V|^2)-2V(\div(V))].
\end{equation}
\end{remark}

\begin{example}
On the $3$-dimensional manifold $M=\{(x,y,z)\in\mathbb{R}^3, z>0\}$, where $(x,y,z)$ are the standard coordinates in $\mathbb{R}^3$, with the Riemannian metric
$$g:=\frac{1}{z^2}(dx\otimes dx+dy\otimes dy+dz\otimes dz),$$ the pair $(V=\frac{\partial}{\partial z}, \lambda=-\frac{1}{z}-2)$ defines an expanding gradient almost Ricci soliton. Precisely, $V=\grad(f)$, for $f(x,y,z)=-\frac{1}{z}$ and one can check that $|V|^2=\frac{1}{z^2}$, $V(|V|^2)=-\frac{2}{z^3}$, $\Delta(|V|^2)=\frac{8}{z^2}$, $|\nabla V|^2=\frac{3}{z^2}$, $\div(V)=-\frac{3}{z}$, $V(\div(V))=\frac{3}{z^2}$ therefore, $\lambda=-\frac{1}{z}-2$ is obtained from (\ref{e5}), without making all the laborious curvature computations.
\end{example}

\begin{example}
On the $3$-dimensional manifold $M=\{(x,y,z)\in\mathbb{R}^3, z>1\}$, where $(x,y,z)$ are the standard coordinates in $\mathbb{R}^3$, with the Riemannian metric
$$g:=e^{2z}(dx\otimes dx+dy\otimes dy)+dz\otimes dz,$$
the pair $(V=e^z\frac{\partial}{\partial z}, \lambda=e^z-2)$ defines a shrinking gradient almost Ricci soliton. Precisely, $V=\grad(f)$, for $f(x,y,z)=e^z$ and one can check that $|V|^2=e^{2z}$, $V(|V|^2)=2e^{3z}$, $\Delta(|V|^2)=8e^{2z}$, $|\nabla V|^2=3e^{2z}$, $\div(V)=3e^z$, $V(\div(V))=3e^{2z}$ therefore, $\lambda=e^z-2$ is immediately obtained from (\ref{e5}).
\end{example}

If $\lambda$ is a constant and $V=\grad(f)$ is a gradient vector field of constant length, from (\ref{13}) we get $\Ric(V,V)=0$ and from (\ref{15}), we obtain $|\nabla V|^2+V(\div(V))=0$. Then Theorem \ref{p1} implies $\lambda=\frac{1}{n-1}\div(V)$, which by integration, in the compact case, give $\lambda=0$, therefore, $\div(V)=0$. In this case,
$f$ is a harmonic function, hence it is constant (by the maximum principle), and $V=0$. Therefore, we can state:

\begin{theorem}
For $n\geq 3$, there exists no non-trivial compact gradient Riemann soliton $(V,\lambda)$ with $V$ of constant length.
\end{theorem}

Also we obtain:

\begin{proposition}
Let $(V,\lambda)$ define a gradient Riemann soliton on an $n$\text{-}di\-men\-sio\-nal compact Riemannian manifold $(M,g)$, $n\geq 3$. Then
$$\lambda\cdot\vol=\frac{n-2}{2(n-1)}\int_M \frac{V(|V|^2)}{|V|^2}d\mu_g.$$
\end{proposition}

If $V=\grad{(f)}$, equation (\ref{4}) becomes
\begin{equation}\label{e1}
\Hess(f)+\frac{1}{n-2}\Ric=\frac{(n-1)\lambda-\Delta(f)}{n-2}g.
\end{equation}

By tracing (\ref{e1}), we get
\begin{equation}\label{e3}
\scal=n(n-1)\lambda-2(n-1)\Delta(f)
\end{equation}
which by differentiating gives
\begin{equation}\label{e4}
d(\Delta(f))=\frac{n}{2}d\lambda -\frac{1}{2(n-1)}  d(\scal)
\end{equation}
and by taking the gradient of (\ref{e3})
\begin{equation}\label{e6}
\grad(\Delta(f))=\frac{n}{2}\grad(\lambda) -\frac{1}{2(n-1)}  \grad(\scal).
\end{equation}

Applying the divergence operator to (\ref{e1}) and using $\div(\Ric)=\frac{1}{2}d(\scal)$, we obtain
\begin{equation}
\div(\Hess(f))=\frac{n-1}{n-2}d\lambda-\frac{1}{n-2}d(\Delta(f))-\frac{1}{2(n-2)}d(\scal).
\end{equation}

Using the relation proved in \cite{blag}, namely
\begin{equation}
\div(\Hess(f))=d(\Delta(f))+i_{Q(\grad(f))}g,
\end{equation}
where $i$ denotes the interior product and $Q$ is the Ricci operator, $g(QX,Y):=\Ric(X,Y)$, we get
\begin{equation}\label{e10}
d(\Delta(f))=d\lambda-\frac{1}{2(n-1)}d(\scal)-\frac{n-2}{n-1}i_{Q(\grad(f))}g.
\end{equation}

Equating (\ref{e4}) and (\ref{e10}), we obtain:
\begin{proposition}
If $(V=\grad(f),\lambda)$ defines a gradient almost Riemann soliton on the $n$\text{-}di\-men\-sio\-nal Riemannian manifold $(M,g)$, $n\geq 3$, then
$$\grad(\lambda)=-\frac{2}{n-1}Q(\grad(f)).$$
\end{proposition}

In equation (\ref{4}), by taking the scalar product with $\Ric$, we get
$$\langle\pounds _{V}g,\Ric\rangle=\frac{1}{n-2}\{(n-1)[n(n-1)\lambda^2-(3n-2)\lambda \cdot \div(V)+2(\div(V))^2]-|\Ric|^2\}$$
and by taking the scalar product with $\pounds _{V}g$, we get
$$\langle\pounds _{V}g,\Ric\rangle=\frac{n-2}{4}\left\{\frac{4[(n-1)\lambda\cdot \div(V)-(\div(V))^2]}{n-2}-|\pounds _{V}g|^2\right\}.$$

If $V$ is a solenoidal vector field, then
$$\frac{1}{n-2}[n(n-1)^2\lambda^2-|\Ric|^2]=-\frac{n-2}{4}|\pounds _{V}g|^2,$$
which implies
$$|\Ric|^2\geq n(n-1)^2\lambda^2.$$

But $|\pounds _{V}g|^2=4|\nabla V|^2$, which gives
$$|\Ric|^2=n(n-1)^2\lambda^2+(n-2)^2|\nabla V|^2.$$

If we assume that $V$ is a solenoidal vector field of gradient type, then (\ref{15}) becomes
$$|\nabla V|^2=\frac{1}{2}\Delta(|V|^2)-\Ric(V,V),$$
therefore, using (\ref{13}), we obtain:

\begin{proposition}\label{p4}
Let $(V,\lambda)$ define a gradient almost Riemann soliton on the $n$\text{-}di\-men\-sio\-nal Riemannian manifold $(M,g)$, $n\geq 3$. If $V$ is a solenoidal vector field, then
$$|\Ric|^2=n(n-1)^2\lambda^2+\frac{(n-2)^2}{2}[\Delta(|V|^2)+(n-1)V(\lambda)].$$
\end{proposition}

From (\ref{14}) and Proposition \ref{p4}, we get:
\begin{corollary}
If $(V,\lambda)$ defines a gradient almost Riemann soliton on the $n$\text{-}di\-men\-sio\-nal Riemannian manifold $(M,g)$, $n\geq 3$, with $V$ a solenoidal vector field, then
$$|\Ric|^2=n(n-1)^2\lambda^2-(n-1)(n-2)^2|V|^2\lambda+\frac{(n-2)^2}{2}[\Delta(|V|^2)+(n-2)V(|V|^2)].$$

Moreover, if $V$ is a unitary vector field, then $M$ is a Ricci-flat manifold.
\end{corollary}

Remark that Petersen \cite{pe} called $f$ a \textit{distance function} if it is a solution of the Hamilton-Jacobi equation $|\grad_g(f)|^2_g=1$, which he uses in his book.

\subsection{Solitons with torse-forming potential vector field}

Assume next that $V$ is a torse-forming vector field, i.e. $\nabla V=aI+\psi\otimes V$, with $a$ a smooth function and $\psi$ a $1$-form, where $\nabla$ is the Levi-Civita connection of $g$ and $I$ is the identity endomorphism on the space of vector fields.
Then the Lie derivative of $g$ in the direction of $V$ is
\begin{equation}\label{h}
\pounds _{V}g=2ag+\psi\otimes \theta+\theta\otimes \psi,
\end{equation}
where $\theta$ is the $g$-dual $1$-form of $V$.
Also $\div(V)=na+\psi(V)$ and (\ref{4}) becomes
$$\frac{1}{2}\pounds _{V}g+\frac{1}{n-2}\Ric=\frac{[(n-1)\lambda-na-\psi(V)]}{n-2}g$$
and replacing $\pounds _{V}g$ from (\ref{h}), we get
\begin{equation}\label{7}
\Ric=[(n-1)(\lambda-2a)-\psi(V)]g-\frac{n-2}{2}(\psi\otimes \theta+\theta\otimes \psi),
\end{equation}
$$Q=[(n-1)(\lambda-2a)-\psi(V)]I-\frac{n-2}{2}(\psi\otimes V+\theta\otimes \zeta),$$
hence
$$\scal=(n-1)[n(\lambda-2a)-2\psi(V)],$$
where $\zeta$ is the $g$-dual of $\psi$.

From (\ref{3}), for $\nabla V=aI+\psi\otimes V$, we also obtain
$$R(X,Y)V=J(X)Y-J(Y)X+[(\nabla_X\psi)Y-(\nabla_Y\psi)X]V,$$
for any $X,Y\in \mathfrak{X}(M)$, where $J:=da-a\psi$.

Moreover, if $\psi$ is a Codazzi tensor field, i.e. $(\nabla_X\psi)Y=(\nabla_Y\psi)X$, for any $X,\linebreak[3] Y\in \nolinebreak\mathfrak{X}(M)$, then
the Jacobi operator $R(\cdot,V)V$ w.r.t. $V$ is given by
$$R(\cdot,V)V=da\otimes V-V(a)I-a[\psi\otimes V-\psi(V)I]$$
and
\begin{equation}\label{5}
\func{Ric}(V,V)=(1-n)[V(a)-a\psi(V)].
\end{equation}

But if $(V,\lambda)$ defines an almost Riemann soliton with torse-forming potential vector
field $V$, then
\begin{equation}\label{6}
\func{Ric}(V,V)=(n-1)[(\lambda-2a)-\psi(V)]|V|^2,
\end{equation}%
and we can state:

\begin{proposition}
\label{p} If $(V,\lambda)$ defines an almost Riemann
soliton on the $n$-di\-men\-sional Riemannian manifold $(M,g)$, $n\geq 3$, such that the potential vector field $V$ is torse-forming and $\psi$ is a Codazzi tensor field, then
$$\lambda=2a+\frac{1}{|V|^2}[(a+|V|^2)\psi(V)-V(a)].$$

Moreover, if $V$ is a concircular vector field, i.e. $\nabla V=aI$, with $a$ a non zero constant, then

i) $\lambda=2a$;

ii) $M$ is a Ricci-flat manifold (i.e. $\Ric=0$).
\end{proposition}

\begin{proof}
From (\ref{5}) and (\ref{6}) we get the expression of $\lambda$.

If $V$ is concircular and $a$ is a non zero constant, then
$\lambda=2a$ and from (\ref{7}) we get $\Ric=0$.
\end{proof}

\begin{remark}
i) If $V$ is concircular, i.e. $\nabla V=aI$, with $a$ a smooth function, $a(x)\neq 0$, for any
$x\in M$, then
\begin{equation*}
V=\frac{1}{2a}\func{grad}(|V|^{2})
\end{equation*}%
whose divergence is $\func{div}(V)=na$.

ii) If $(V,\lambda)$ defines an almost Riemann soliton on $(M,g)$, $n\geq 3$, $V$ is torse-forming and $g$-orthogonal to $\zeta$, and $\psi$ is a Codazzi tensor field, then $M$ is an almost quasi-Einstein manifold with associated functions $$\left(-(n-1)\frac{V(a)}{|V|^2},-\frac{n-2}{2}\right).$$
\end{remark}

Next we shall find a condition such that an almost Ricci soliton to provide an almost Riemann soliton. Recall that the conharmonic curvature tensor $\mathcal{H}$ was defined in \cite{po}
\begin{equation}\label{e32}
\mathcal{H}(X,Y)Z:=R(X,Y)Z$$
$$+\frac{1}{\dim(M)-2}[g(Z,X)QY-g(Y,Z)QX+\Ric(Z,X)Y-\Ric(Y,Z)X],
\end{equation}
for any $X$, $Y$, $Z\in \mathfrak{X}(M)$.

\begin{proposition}\label{p3}
Let $(V,\lambda)$ define an almost Ricci soliton on the $n$\text{-}di\-men\-sio\-nal Riemannian manifold $(M,g)$, $n\geq 3$. Then $(V,2\lambda)$ defines an almost Riemann soliton if and only if the conharmonic curvature tensor field $\mathcal{H}$ of $g$ satisfies $\mathcal{H}=\frac{n-3}{n-2}R$.
\end{proposition}
\begin{proof}
Applying the Kulkarni-Nomizu product with $g$ to the Ricci soliton equation (\ref{1}), we deduce that $\Ric\odot g=R$ if and only if
$$R(X,Y)Z=g(Y,Z)QX-g(X,Z)QY+\Ric(Y,Z)X-\Ric(X,Z)Y,$$
for any $X,Y,Z\in \mathfrak{X}(M)$, which gives the conclusion.

\end{proof}

\begin{corollary}
If $(V,\lambda)$ defines an almost Ricci soliton on the $n$\text{-}di\-men\-sio\-nal Riemannian manifold $(M,g)$, $n\geq 3$, with $V$ a concircular vector field, then $(V,2\lambda)$ defines an almost Riemann soliton if and only if
$$R(X,Y)Z=2(\lambda-a)[g(Y,Z)X-g(X,Z)Y],$$
for any $X,Y,Z\in \mathfrak{X}(M)$.
Moreover, we get

i) $V(a)\theta=|V|^2da$;

ii) $\grad(a)=V(a)\frac{V}{|V|^2}$;

iii) $M$ is of constant scalar curvature.
\end{corollary}
\begin{proof}
If $V$ is concircular, from the Ricci soliton equation (\ref{1}) we get
$$\Ric=(\lambda-a)g, \ \ Q=(\lambda-a)I, \ \ \scal=n(\lambda-a).$$
Then from Proposition \ref{p3}, we deduce that the Riemann soliton equation (\ref{3}) is satisfied for $(V,2\lambda)$ if and only if
$$R(X,Y)Z=2(\lambda-a)[g(Y,Z)X-g(X,Z)Y],$$
for any $X,Y,Z\in \mathfrak{X}(M)$. Since
$$R(X,V)V=X(a)V-V(a)X,$$
for any $X\in \mathfrak{X}(M)$, we have
$$[2(\lambda-a)|V|^2+V(a)]X=[2(\lambda-a)\theta(X)+X(a)]V,$$
where $\theta=i_Vg$, and applying $\theta$, we get i) and ii).
Also, taking into account that $\div(\Ric)=\frac{1}{2}d(\scal)$, we get $d\lambda=da$, hence $d(\scal)=0$ and we deduce iii).

\end{proof}

Assume that the potential vector field $V$ of the Riemann soliton defined by $(V,\lambda)$ is a torse-forming vector field, i.e. $\nabla V=aI+\psi\otimes V$, with $a$ a smooth function and $\psi$ a $1$-form, where $\nabla$ is the Levi-Civita connection of $g$. Let $\theta=i_Vg$ and $i_{\zeta}g=\psi$.
Notice that
$$\nabla \theta=ag+\psi\otimes \theta,$$
$$(\nabla_V\psi)V+(\nabla_V\theta)\zeta=V(\psi(V)), \ \ (\nabla_{\zeta}\psi)V+(\nabla_{\zeta}\theta)\zeta=\zeta(\psi(V)),$$
$$(\nabla_V\theta)\zeta\cdot |\zeta|^2=(\nabla_{\zeta}\theta)\zeta\cdot \psi(V).$$
\pagebreak
In particular, if $V$ and $\zeta$ are $g$-orthogonal, then $a=-\frac{1}{2|V|^2}V(|V|^2)$.

But differentiating covariantly (\ref{7}), we get
$$
(\nabla_X\Ric)(Y,Z)=\{(n-1)[X(\lambda)-2X(a)]-X(\psi(V))\}g(Y,Z)
$$
\begin{equation}\label{44}
-\frac{n-2}{2}[\psi(Y)(\nabla_X\theta)Z+\psi(Z)(\nabla_X\theta)Y+\theta(Y)(\nabla_X\psi)Z+\theta(Z)(\nabla_X\psi)Y],
\end{equation}
for any $X,Y,Z\in \mathfrak{X}(M)$ and we can state:
\begin{proposition}
If $(V,\lambda)$ defines an almost Riemann soliton with concircular potential vector field $V$, then

i) $M$ is Ricci symmetric (i.e. $\nabla \Ric=0$) if and only if $d\lambda=2da$;

ii) the Ricci tensor field $\Ric$ is $\theta$-recurrent (i.e. $\nabla \Ric =\theta\otimes \Ric$) if and only if $$\grad(\lambda-2a)=(\lambda-2a)V;$$

iii) $\Ric$ is a Codazzi tensor field (i.e. $(\nabla_X\Ric)(Y,Z)=(\nabla_Y\Ric)(X,Z)$, for any $X,Y,Z\in \mathfrak{X}(M)$) if and only if $$d(\lambda-2a)\otimes I=I\otimes d(\lambda-2a);$$

iv) $M$ has cyclic Ricci tensor (i.e. $(\nabla_X\Ric)(Y,Z)+(\nabla_Y\Ric)(Z,X)\linebreak[2]+(\nabla_Z\Ric)(X,Y)\linebreak[0] =0$, for any $X,Y,Z\in \mathfrak{X}(M)$) implies
$$\grad(\lambda-2a)=-2V(\lambda-2a)\frac{V}{|V|^2}.$$
\end{proposition}
\begin{proof}
If $V$ is concircular, from (\ref{44}) we get
$$(\nabla_X\Ric)(Y,Z)=(n-1)[X(\lambda)-2X(a)]g(Y,Z),$$
for any $X,Y,Z\in \mathfrak{X}(M)$ and we easily get the conclusions.

\end{proof}

Next we shall assume the curvature condition $R(V,\cdot)\cdot \Ric=0$, where by $\cdot$ we denote the derivation of the tensor algebra at each point of the tangent space.

\begin{proposition}\label{p2}
Let $(V,\lambda)$ define an almost Riemann soliton on the $n$\text{-}di\-men\-sio\-nal Riemannian manifold $(M,g)$, $n\geq 3$, with torse-forming potential vector field $V$. If $R(V,\cdot)\cdot \Ric=0$, then
$$\lambda=\frac{1}{2}\cdot\frac{V(|V|^2)}{|V|^2}+a, \ \text{or} \ a=\frac{1}{2}\cdot\frac{V(|V|^2)}{|V|^2}\pm |V|\cdot |\zeta|.$$
\end{proposition}
\begin{proof}
If $V$ is torse-forming, from (\ref{3}) we get
$$2R(X,Y)Z=[2(\lambda-2a)g(Y,Z)-\psi(Y)\theta(Z)-\theta(Y)\psi(Z)]X$$
$$-[2(\lambda-2a)g(X,Z)-\psi(X)\theta(Z)-\theta(X)\psi(Z)]Y$$
$$-[g(Y,Z)\psi(X)-g(X,Z)\psi(Y)]V-[g(Y,Z)\theta(X)-g(X,Z)\theta(Y)]\zeta$$
and, in particular,
\begin{equation}\label{8}
2R(V,Y)Z=[2(\lambda-2a)-\psi(V)][g(Y,Z)V-\theta(Z)Y]-\psi(Z)[\theta(Y)V-|V|^2Y]
$$$$-[|V|^2g(Y,Z)-\theta(Y)\theta(Z)]\zeta.
\end{equation}

The condition that must be satisfied by $\Ric$ is:
\begin{equation}\label{88}
\Ric(R(V,X)Y,Z)+\Ric(Y,R(V,X)Z)=0,
\end{equation}
for any $X$, $Y$, $Z\in \mathfrak{X}(M)$.

Replacing the expression of $R(V,\cdot)\cdot$ from (\ref{8}) in (\ref{88}) we get:
$$
[2(\lambda-2a)-\psi(V)][g(X,Y)\Ric(V,Z)-\theta(Y)\Ric(X,Z)$$$$+g(X,Z)\Ric(V,Y)-\theta(Z)
\Ric(Y,X)]$$$$
-\psi(Y)\theta(X)\Ric(V,Z)+\psi(Y)|V|^2\Ric(X,Z)$$$$-\psi(Z)\theta(X)\Ric(V,Y)+\psi(Z)|V|^2\Ric(X,Y)$$
$$-[|V|^2g(X,Y)-\theta(X)\theta(Y)]\Ric(\zeta,Z)-[|V|^2g(X,Z)-\theta(X)\theta(Z)]\Ric(\zeta,Y)=0,$$
for any $X$, $Y$, $Z\in \mathfrak{X}(M)$.

For $Y=Z:=V$ we have:
$$[\lambda-2a-\psi(V)][\theta(X)\Ric(V,V)-|V|^2\Ric(X,V)]=0.$$

But from (\ref{7}):
$$
\Ric(X,V)=\left[(n-1)(\lambda-2a)-\frac{n}{2}\psi(V)\right]\theta(X)-\frac{n-2}{2}|V|^2\psi(X),
$$
which implies
$$[\lambda-2a-\psi(V)][\psi(V)\theta(X)-|V|^2\psi(X)]=0.$$

Then either $\psi(V)=\lambda-2a$ or
$\psi(V)V=|V|^2\zeta$.

But $\psi(V)=\frac{1}{2|V|^2}V(|V|^2)-a$, hence the conclusion.

\end{proof}

\bigskip

\textit{Adara M. Blaga}

\textit{Department of Mathematics}

\textit{West University of Timi\c{s}oara}

\textit{Bld. V. P\^{a}rvan nr. 4, 300223, Timi\c{s}oara, Rom\^{a}nia}

\textit{adarablaga@yahoo.com}

\end{document}